\documentclass[russian,english]{article}
\usepackage{amsthm,amsfonts, amsbsy, amssymb,amsmath,graphicx}
\usepackage{graphics}
\usepackage{babel}
\newtheorem{thm}{Theorem}
\newtheorem{lem}{Lemma}
\newtheorem{re}{Remark}

\newtheorem{st}{Statement}
\newtheorem{cj}{Conjecture}

\newcounter{tdfn}
\newcounter{trk}
\newenvironment{rk}
{\vspace{0.15cm}{\bf Remark \arabic{trk}.}} {\par
\addtocounter{trk}{1}}

\def\R{{\mathbb R}}
\def\Q{{\mathbb Q}}
 \def\Z{{\mathbb Z}}
 \def\0{{\mathbbf 0}}
 \def\1{{\mathbbf 1}}

 \def\N{{\mathbb N}}

\title{On The Chromatic Numbers of Integer and Rational Lattices}

\author{Vassily Olegovich Manturov
\footnote{Peoples' Friendship University of Russia, vomanturov at
yandex.ru} \footnote{The  author  was partially supported by grants
of the Russian Government 11.G34.31.0053, RF President NSh
1410.2012.1, Ministry of Education and Science of the Russian
Federation 14.740.11.0794}}

\date{}

\begin{document}

\maketitle

\begin{abstract}
In the present paper, we give new upper bounds for the chromatic
numbers for integer lattices and some rational spaces and  other
lattices. In particular, we have proved that for any concrete
integer number $d$, the chromatic number of $\Z^{n}$ with critical
distance $\sqrt{2d}$ has a polynomial growth in $n$ with exponent
less than or equal to $d$ (sometimes this estimate is sharp). The
same statement is true not only in the Euclidean norm, but also in
any $l_{p}$ norm. Besides, we have given concrete estimates for some
small dimensions as well as upper bounds for the chromatic number of
$\Q_{p}^{n}$, where by $\Q_{p}$ we mean the ring of all rational
numbers having denominators not divisible by some prime numbers.
\end{abstract}



\section{Introduction}

By the {\em chromatic number} of a metric space $M$ with {\em
forbidden distance} (or {\em critical distance}) $d$ we mean the
minimal cardinality of a set $S$ for which there exists a map
$f:M\to S$ such that for every two points $x,y\in M$ at  distance
$d$ we have $f(x)\neq f(y)$. {\bf Notation:} $\chi(M,d)$. For
shorthand we write $\chi(M)$ for $\chi(M,1)$.

The chromatic numbers of Euclidean spaces and linear spaces over the
rational numbers (if the norm is Euclidean, we denote them by
$\Q^{n}$) were studied by many authors, see, e.g.,
\cite{Rai,BP,LR,Cibulka,Kup,BMP} and references therein.

The chromatic numbers for integer lattices in $l_{2}$ and $l_{1}$
norms were studied, in particular, by Z.F\H{u}redi and J.Kang
\cite{FK}, where a lower bound exponential in $n$ was found for
$\chi(\Z^{n},\sqrt{r})$ for even $r$ in $l_{2}$-norm, and similarly,
for $\chi(\Z^{n},\sqrt{r})$ for even $r$ in $l_{1}$-norm, however,
the result really proved dealt with some special case when $r$
depends on $n$ (e.g., $r=2q,n=4q-1$ for some integer $q$). Some
better estimates were obtained in \cite{Rai2}.

It turns out that the best known upper asymptotic estimates for the
chromatic number of rational spaces are exactly those known ones for
the Euclidean spaces: the chromatic number of $\Q^{n}$ is known to
be bounded from above by $(3+o(1))^{n}$ as $n$ tends to infinity
\cite{LR,Rai}.

In the present paper, we undertake a systematic treatment of integer
lattices and some variations of them: lattices over some completions
of the ring of integers, lattices over rational numbers with
denominators not divisible by some concrete numbers.

The paper is organised as follows.

In the next section, we deal with low-dimensional cases of integer
lattices, in particular we find cases when the chromatic number is
equal to $3$. These results were not previously found in the
literature.

The main result of our paper goes in Section 3: we prove that for a
fixed $m$, the chromatic number $\chi(\Z^{n},\sqrt{2m})$ is
estimated from above by $c\cdot n^{m}$ in any norm $l_{p}$, where
$c$ does not depend on $n$.

This result goes in contrast with the similar result concerning
rational lattices because for the latter, it is known that the {\em
lower} bound is exponential.

Then we revisit some known lower bounds coming from the
Frankl--Wilson theorem, which can be used for obtaining lower bounds
for integer lattices.

Another interesting case deals with estimates for the chromatic
numbers of rational spaces. We give new upper estimates for lattices
over rings of rational numbers whose denominators are coprime with
$5$ and $3$, respectively (Theorems \ref{5grid} and \ref{3grid}).

As a step towards new estimates for rational lattices, we consider
lattices over rational numbers with some forbidden denominators.

The paper is concluded by further discussion and open problems.

\subsection{Acknowledgements}

I am grateful to A.M.Raigorodskii and A.B.Kupavskii for various
fruitful discussions of the subject, to I.D.Shkredov for drawing my
attention to the papers \cite{OB,Ruzsa1,Ruzsa2}, and to J.-H.Kang
for a discussion of the paper \cite{FK}.

The final version of this text was written during my stay in the
Mathematical Sciences Center, Tsinghua University, Beijing, People's
Republic of China. I am very grateful to the Center for the creative
research atmosphere.

\section{Low-dimensional integer lattices}

The following theorem is evident, see, e.g., \cite{BP}.
\begin{thm}
For every  $k$ which can be represented as a sum of two integer
squares, one has $\chi(\Z^{2},\sqrt{k})=2$. Otherwise,
$\chi(\Z^{2},\sqrt{k})=1$. Moreover, the same statement is true for
$\Z^{n}$.
\end{thm}

For $\Z^{3}$, it makes sense to consider only critical distances of
type $\sqrt{4l+2},l$ is odd: in the case of odd distances we know
that the chromatic number does not exceed two by parity reasons. As
for the case when we have $4 l,l\in \Z$, under the square root, it
is reduced in $\R^{3}$ to the case of $\sqrt{l}$ since every
representation of $4l$ as a sum of three squares of integers
consists of three even squares.

\begin{thm}[Upper estimate: the universal colouring]
For every $k=4l+2,l\in \Z$, we have $\chi(\Z^{3},\sqrt{k})\le 4$.
\end{thm}

\begin{proof}
We consider sets of points of the three-dimensional lattice where
the sum of the three coordinates is even. The colouring of the other
points will be obtained by shifting this colouring by a vector
$(1,0,0)$.

Let us consider the following ``universal $4$-colouring'' with four
colours $(0,0),(0,1),(1,0),(1,1)$, where with a triple of integer
coordinates in $\Z^{3}$ we associate two numbers, the first being
the parity of the first coordinate, and the second being the parity
of the third coordinate.

It is clear that if two points in the integer $3$-dimensional
lattice are at distance $\sqrt{4l+2}$, then either they have
different parity of the $z$ coordinate, or they have different
parity of the $x$ coordinate.
\end{proof}

To get lower estimates we shall often use the {\em Raiskii--Moser
spindle} \cite{MM},\cite{Raiskii}, and its generalisations. By the
two--dimensional spindle (for critical distance $1$) we mean a graph
on $7$ vertices which looks as follows: one vertex $A$ forms two
unit-distance triangles $ABC$ and $AB'C'$, there are also two
unit-distance triangles $BCD$ and $B'C'D'$ for some points $D$ and
$D'$ which are at distance $1$. When trying to colour this spindle
with three colours, we get to a contradiction: $A$ and $D$ have to
have the same colour, the same is true for $A$ and $D'$. On the
other hand, $D$ and $D'$ are at  distance $1$, so, their colours
should be different. One can take the homothety of this spindle for
any critical distance, the resulting graph embedding will be called
{\em spindle} as well.

In dimension $n$, instead of triangles, one takes two pairs of
unit-distance $n$-simplices, say, $(A,A_{1},\dots, A_{n})$ and
$(A,A'_{1},\dots, A'_{n})$ with a common point $A$ and takes the
points $B$ and $B'$ to be at the unit distance from all
$A_{i},i=1,\dots, n$ and from all $A'_{i},i=1,\dots, n$,
respectively. Obviously, if $B$ and $B'$ are at distance $1$ then
this graph admits no proper $(n+1)$-colouring, which was first used
to prove $\chi (\R^{n},1)\ge n+2$.

In rational spaces or integer lattices for concrete critical
distances, one usually can not find spindles directly. Nevertheless,
the argument can be corrected if one considers the {\em generalised
spindle} (or {\em Kupavskii spindle}). Assume we have two pairs of
triangles $(A,B,C,D), (A,B',C',D')$ for some critical distance as
above (in spaces of higher dimensions, two pairs of simplices) and
that the points $D$ and $D'$ are not at a critical distance. Denote
the distance $l(D,D')$ by $d$. Assume that for every proper
colouring there are two points ${\tilde D}, {\tilde D'}$ of
different colours and an isometry of the space which takes $D\mapsto
{\tilde D}, D'\mapsto {\tilde D}'$. Then taking the image of all the
points $(A,B,C,D,B',C',D')$, we see that there is no way to colour
them with $3$ colours.

This statement (in some more generality) for the case of Euclidean
spaces was proved by A.B.Kupavskii \cite{Kup}.
 Certainly, isometries in the integer
case have to be treated in a more delicate way than those in the
real case.

Now, to prove the lower estimates for $\Z^{3}$ we shall often use
the generalized spindle (Kupavskii spindle) construction. Note that
in $\Z^{3}$, the equality of distances $l(D,D')$ and $l({\tilde
D},{\tilde D}')$ does not guarantee the existence of such an
isometry taking $D$ to ${\tilde D}$ and $D'$ to ${\tilde D'}$.

To this end, we shall need the following
\begin{lem}[Integral Analogue of the Kupavskii Lemma]
Let $A=(a_{1},a_{2},a_{3})$ be a point in $\Z^{3}$, so that
$a_{1}+a_{2}+a_{3}$ is even and $GCD(a_{1},a_{2},a_{3})=1$. Assume
further, $m=b_{1}^{2}+b_{2}^{2}+b_{3}^{2}$ is even for some integers
$b_{1},b_{2},b_{3}$ with $GCD(b_{1},b_{2},b_{3})=1$.
 Then for every proper colouring (with at least two distinct colours) of $\Z^{3}$ with forbidden distance $\sqrt{m}$,
 there exist $P,Q\in \Z^{3}$ and an isometry of
$\Z^{3}$ which takes the origin to $P$ and $A$ to $Q$ such that $P$
and $Q$ have different colours.  \label{fundlemma}
\end{lem}

\begin{proof}
Indeed,  there is a chain $X_{0},X_{1},\dots, X_{k}$ of points in
$\Z^{3}$ from $(0,0,0)=X_{0}$ to $(b_{1},b_{2},b_{3})=X_{k}$ with
all $X_{i},i=1,\dots, k-1$ such that for every two adjacent
$X_{i},X_{i+1}$ the vector $X_{i+1}-X_{i}$ is obtained from the
vector $(a_{1},a_{2},a_{3})$ by an isometry of $\Z^{3}$. Now, since
$X_{0}$ and $X_{k}$ are colored differently, there exist two
adjacent $X_{i}$ and $X_{i+1}$ having different colours.

The existence of such a chain is left for the reader as an
exercise.\label{exrf}
\end{proof}

\begin{thm}
$\chi(\Z^{3},\sqrt{2})=4$.
\end{thm}

One can easily construct the spindle: we take the two triangles
$(0,0,0),(0,1,1),(1,0,1),(1,1,2)$ and a similar pair of triangles
which is obtained from the first pair by permuting the second and
the third coordinates. The points $(1,1,2)$ and $(1,2,1)$ are at the
distance $\sqrt{2}$.

Examples of forbidden distances for which the chromatic number in
the three-dimensional space is equal to four were well known, see,
e.g., \cite{BP}. M.Benda and M.Perles in \cite{BP} asked the
question whether there exists a forbidden distance in $\Q^{3}$ for
which the chromatic number is equal to $3$.

\begin{thm}
For $k=10+12 l,l\in \Z$, one has $\chi(\Z^{3},\sqrt{k})=3$.
\end{thm}

\begin{proof}
One can easily check that any decomposition of $10+12l,l\in \Z$ into
a sum of three integer squares looks like $a^{2}+b^{2}+c^{2}$, where
modulo $6$ reduction of the triple $(a,b,c)$ coincides with one of
the following triples (up to order):
$(1,3,0),(5,3,0),(3,3,2),(3,3,4)$.

Then we colour the points with even sum of coordinates as follows:
for the point with coordinates $x,y,z$ we take the residue classes
of $x+y+z$ modulo $6$, which provides a three-colouring. Analogously
one gets a three-colouring for the set of those points whose sum of
coordinates is odd.
\end{proof}

Now, we turn to those forbidden distances for $\Z^{3}$ for which the
chromatic number is equal to four.

\begin{thm}
If $m=a^{2}+ab+b^{2}$ for some integers $a,b$ then
$\chi(\Z^{3},\sqrt{2m})=4$. In particular, let $p=6k+1$ be a prime
number for an integer $k$. Then $\chi(\Z^{3},\sqrt{2p})=4$.
\end{thm}

\begin{proof}
Indeed, assume first $m=a^{2}+ab+b^{2}$ for coprime $a,b$.

We suppose our space is $3$-colourable and want to get a
contradiction.
 Let us first assume exactly one of $a$ and $b$ is odd,
without loss of generality, assume $a$ is odd, $b$ is even.

 We have $2m=(a^{2}+b^{2}+(a+b)^{2})$.
So, the distance $\sqrt{2m}$ is realised by vectors with three
coordinates, whose difference has coordinates equal to $\pm a$, $\pm
b$, $\pm (a+ b)$ up to order. So, we shall try different vectors to
construct the generalised integer spindle.

First, take the two triangles $ABC,BCD$ with the following vertices
$A=(0,0,0),B=(a,b,a+b),C=(-b,a+b,a),(a-b,a+2b,2a+b)=D$.

Now, we may get other pairs of triangles with $A=(0,0,0)$ by
permuting the coordinates and taking minus sign for $a$ and/or $b$.
For example, there is a pair of triangles with the free end
$D'=(-a-b,-2a+b,-a+2b)$. Now, it is easy to see that
$GCD(-a-b,-2a+b,-a+2b)$ is either $1$ or $3$. If it is $3$, then by
changing $b$ to $-b$, we get $GCD(-a+b,-2a-b,-a-2b)=1$, and we get
the desired spindle.

In the case when both $a,b$ are odd, we notice that the same pair of
triangles can be obtained starting with $(a,a+b)$, where $(a+b)$ is
even.

In the case when $a,b$ are not coprime, take $c=GCD(a,b),
a=a'c,b=b'c$ and construct analogous spindles for the sublattice
with all coordinates divisible by $c$.

Now we apply Lemma \ref{fundlemma} and see that after applying some
isometry to $\Z^{3}$,  the images ${\tilde D}$ and ${\tilde D'}$
will have different colours. Taking all images of $A,B,C,D,B',C',D'$
we get a contradiction to $3$-colouring of the space.
\end{proof}

Collecting the above results about colourings of $\Z^{3}$, we get
the following
\begin{thm}
We have:

\begin{enumerate}

\item  $\chi(\Z^{3},\sqrt{m})=2$ if and only if $m$ is odd;

\item for even $m$ we have $\chi(\Z^{3},\sqrt{m})$ is either $3$ or
$4$;

\item If $m\equiv 10 ($\em{mod}  $12)$, then
$\chi(\Z^{3},\sqrt{m})=3$;

\item If $m=2  (a^{2}+b^{2}+ab)$, $a,b\in \Z$, then
$\chi(\Z^{3},\sqrt{m})=4$;

\item $\chi(\Z^{3},\sqrt{m})=\chi(\Z^{3},2\sqrt{m})$.

\end{enumerate}

\end{thm}

The only statement of the above theorem, we haven't yet proved, is
2. We prove it in several steps.

a) It suffices to prove it for $m=2p$ for prime $p$.

b) Let $m=2p=a^{2}+b^{2}+c^{2}$ for $CGD(a,b,c)=1$.

c) From exercise on page \pageref{exrf}, it follows that there is a
chain in $\Z^{3}$ from the origin to $(0,1,1)$ with every two
adjacent nodes at distance $\sqrt{m}$.

d) If there is a chain of odd length $l$, then we easily construct
analogous chains from the origin to $(1,0,1)$ and from $(1,0,1)$ to
$(0,1,1)$ which leads to a closed chain of length $3l$ which
contradicts $2$-colourability.

e) Assume the chain from c) has even length. Then, there is a chain
in $\Z^{3}$ of even length (with distance $\sqrt{m}$ between two
adjacent points) from the origin to every point with even sum of
coordinates. In particular, there is a chain of even length from the
origin to $(a+1,b,c+1)$. Thus, there is a chain of odd length from
the origin to $(1,0,1)$. From d) we get a contradiction with
$2$-colourability.

The theorem is proved.

The first critical distance which does not fit into the list above
is $\sqrt{30}$.

\begin{cj}
There are no other examples of the chromatic number $3$, in other
words, $\chi(\Z^{3},\sqrt{m})=3$ only for those  $m$ which can be
represented in the form $2^{2k}\cdot l$, where $l\equiv 10 \mbox {
mod } 12$.\label{cjcj}
\end{cj}

Now, let us pass to the dimensions 4 and 5.

\begin{thm}

We have:

\begin{enumerate}

\item (A.B.Kupavskii)
For $k=4 l+2,l\in \Z$, one has $\chi(\Z^{4},\sqrt{k})\le
4,\chi(\Z^{5},\sqrt{k})\le 8$.

\item $\chi(\Z^{4},\sqrt{8 k})=\chi(\Z^{4},\sqrt{2k})$.

\item $\chi(\Z^{4},\sqrt{4 l})\le 4$ for odd $l$.

\end{enumerate}
\label{t1k1}
\end{thm}

\begin{proof}
To prove the first statement, it suffices to colour the unit cube
$\{0,1\}^{4}$ (resp., $\{0,1\}^{5}$) corresponding to the parities
of the coordinates. Indeed, if two points in $\Z^{4}$ (resp.,
$\Z^{5})$ have the same parity for all coordinates, then the square
of the distance between these points is divisible by four. Besides,
it suffices to colour only one half of the cube with the sum of
coordinates being even (the ``odd'' part of the cube is coloured
similarly). So, in $\Z^{4}$ (in fact, in $\{0,1\}^{4}$) we colour
$8$ points $(a,b,c,d), a,b,c,d\in \Z_{2},a+b+c+d\equiv 0$ modulo
$2$,  with four colours in such a way that every two opposite points
$(x,y,z,t)$ and $(1-x,1-y,1-z,1-t)$ have the same colour. For the
$5$-dimensional case, it suffices to use this four-colouring for the
first four coordinates and add an independent colour representing
the parity of the fifth coordinate: in total, we get an
$8$-colouring.

The second result follows from the fact that the sum of squares of
four integer numbers, at least one of which is odd, is never
divisible by $8$, so, the problem is reduced to the case when all
coordinates are even.

The upper bound in the third case can be obtained as follows. For
the colouring we take two modulo two residue classes. The first
class is equal to the parity of the first coordinate. The second one
is equal to the sum of parities of $[\frac{x_i}{2}]$ over all
$i=1,2,3,4$.
\end{proof}

\begin{rk}
First note that the estimate given above give a {\em universal
covering} for all forbidden distances of one of three types listed
in the formulation of the theorem. The above proof of the first
statement can be generalized for higher dimensions. We shall
consider the question of upper bounds for $\chi(\Z^{n},4k+2)$ in a
separate paper.
\end{rk}

\section{For every $m$ the growth of $\chi(\Z^{n},\sqrt{2m})$ is
polynomial in $n$ of degree at most $m$}

It is well known, see, e.g., \cite{Rai}, that for rational spaces
the lower estimates for the chromatic number grow exponentially as
the dimension tends to infinity. Below we prove that in the case of
integer lattices it is never so for any concrete forbidden distance.

Let us start with a well known theorem about integer lattices, see,
e.g., \cite{BP}. The colouring by a scalar product will be used for
the proof of the main theorem.

The following theorem is well known.
\begin{thm}
The growth of $\chi(\Z^{n},\sqrt{2})$ is linear as $n$ tends to
infinity.
\end{thm}

\begin{proof}
To get a lower bound (see \cite{FK}), let us take  the set of points
in $\Z^{n}$ with exactly one non-zero coordinate equal
 to $\pm 1$  and the other coordinates all equal to zero, then it can not be coloured with less than $n$ colours for any $n\ge 2$.

The upper estimate is established by the following colouring. In
$\Z^{n}$, let us consider the following vector: $v=(1,3,5,7,\dots,
2n-1)$. For every  $u\in \Z^{n}$, let us consider the scalar product
$\langle u,v\rangle$. It is clear that if two integer points
$u_{1},u_{2}$ are at distance  $\sqrt{2}$, we have $\langle
u_{1},v\rangle\neq \langle u_{2},v\rangle$. More precisely, the
difference of values of $\langle \cdot,v\rangle$ is an even number
whose absolute value is between $2$ and $4n-4$. Thus, if we take the
residue class of this scalar product modulo $4n-2$ for the
colouring, we get a proper $(2n-1)$-colouring.
\end{proof}

The idea of colouring by using scalar products taken modulo some
large integer will later be used in some more complicated
situations. In particular, it will be used for our main result, the
polynomial upper bounds for the chromatic number of integer lattices
with a fixed critical distance.

\begin{thm} For every fixed $m$ the upper estimate for $\chi(\Z^{n},\sqrt{2m})$  in any norm $l_{\alpha}$ is polynomial
in $n$ of degree at most $m$. \label{keythn}
\end{thm}

Before proving this general estimate which relies on some deep
additive combinatorics, we shall give an explicit colouring for the
following partial case.

\begin{st}
$\chi(\Z^{n},2)$ grows quadratically as $n\to
\infty$.\label{forfour}
\end{st}

Let us prove the upper estimate. The lower estimate is in fact well
known and  will be proved later.

Let $n$ be an integer number. Let $p$ be a prime such that $p\le
n\le 2p$.

We shall prove the quadratic upper bound for prime $p$ which
obviously yields the quadratic upper bound as $n\to \infty$.
 Consider the set $(k,a^{k} \mbox{ mod } p)$ of $p$ elements from the
abelian group $S=\{0,\dots,p-1\}\times \{0,\dots, p-1\}$:
 where $k$ runs over $\{1,\dots, p-1\}$
and $a$ is a primitive $(p-1)$-th root of unity in $\Z_{p}$.

It can be easily seen that for any four distinct elements $a,b,c,d$
from the subset described above we have $a-b\neq c-d$. Indeed, if
for some $e,f,g,h\in \Z_{p}$ we have $f-e=h-g$ and $h\neq f,e\neq f$
we see that $a^{f}-a^{e}$ differs from $a^{h}-a^{g}$ by
multiplication by $a^{f-h}$.

We have constructed a set (abelian group) with no solution to
$a-b=c-d$ for distinct $a,b,c,d$. Now, we shall modify this set a
little bit to get rid of solutions of some simpler equations.

Now, take the set $S'\subset \Z\times \Z$ of integer numbers
$(4k-3,a^{k} \mbox{ mod
 } p)$ where $a^{k}\;\mbox{ mod }\; p$ is treated as an integer between $0$ and $p-1$
 (we use the inclusion $\Z_{p}\subset \Z$).

\begin{lem}
For every four distinct elements $a,b,c,d \in S'$ we have:

\begin{enumerate}
\item none of the sums $\pm a\pm b\pm c\pm d$ is equal to zero.

\item the absolute value of the first coordinate of the sum $\pm a\pm b\pm
c\pm d$ does not exceed $16 p$, and the absolute value of the second
coordinate does not exceed $4p$.
\end{enumerate}\label{lmch}
\end{lem}

\begin{proof}
The second statement is evident.

We have proved that $a-b=c-d$ for $a,b,c,d \in S'$  implies $a=c$ or
$a=b$. The equation $a+b+c+d=0$ has no solutions because of
positivity of $a,b,c,d$, and $a+b+c-d\neq 0$ follows from a modulo
$4$ argument.
\end{proof}

Now, considering $S'$ as a subset of the abelian group
$S''=\Z_{16p+1}\times \Z_{4p+1}$, we see that for every four
distinct elements $a,b,c,d \in S'\subset S''$ we have $\pm a\pm b\pm
c\pm d\neq 0 \in S''$. We shall call these $p$ elements constituting
the subset $S'\subset S''$ {\em the distinguished elements} of
$S''$. Denote these distinguished elements from $S''$ by
$q_{1},\dots, q_{p}$. They form a vector which will be used to
construct the desired colouring.

Now, for each vector $x=(x_{1},\dots, x_{p})\in \Z^{p}$ let $x'_{j}$
be the mod $p$ residue class of $x_{j}$ considered as an integer.
With $x$ we associate the element (colour) $f(x)=\sum x'_{j}\cdot
q_{j}\in
 S''$ of the group $S''$.

\begin{lem}
If the distance between two points $x,{\tilde x}$ is equal to $2$
then $f(x)\neq f({\tilde x})$ in $S''$.
\end{lem}

\begin{proof}
Indeed, since all elements $q_{i}$ are non-zero, two points having
all coordinates but one equal and one coordinate which differs by
two, get different colours. If two points have all coordinates but
four equal and in each of four coordinates the difference is $\pm
1$, these points have different colours by Lemma \ref{lmch}.
\end{proof}

Now, taking into account that $|S''|$ grows as $n^{2}$ as $n$ tends
to infinity, we get the claim of Statement \ref{forfour}.

Let us now return to the proof of Theorem \ref{keythn}. We shall
prove this theorem for the $l_{2}$-norm. The construction used in
the proof is actually the same for all norms $l_{\alpha}$.

We shall start with the lower estimate. Let ${\cal M}$ be a subset
of integers of cardinality $N=|{\cal M}|$. Fix an integer number
$m$. The following question has been studied by many authors, see,
e.g.\cite{OB},\cite{Ruzsa1,Ruzsa2} and references therein.

Which is the largest cardinality of the subset ${\cal M}'\subset
{\cal M}$ for which there are no non-trivial solutions of the
equation

\begin{equation}
a_{1}+\dots+a_{m}-a_{m+1}-\dots -a_{2m}=0, \label{sidon}
\end{equation}
where $a_{i},i=1,\dots, 2m\in {\cal M}$? What can one say when $N$
tends to infinity?

The answer of course depends on the definition of non-trivial
solution. We shall adopt the definition from \cite{OB} (sets with no
non-trivial solutions to similar linear equations are called {\em
Sidon sets}).

A solution to (\ref{sidon}) is said to be {\em trivial} if there are
exactly $l$ different elements among $a_{j}$, and if we fix one
concrete $a_{j}$ and take all $a_{k}$ not equal to any of $a_{j}$ to
be $0$, we shall still get a solution.

For example, for $m=4$, the solution $a_{1}=a_{3}=1, a_{2}=a_{4}=2$
is trivial, whence $a_{1}=0,a_{2}=2,a_{3}=a_{4}=1$ is not.

In \cite{OB} the following statement is proved
\begin{st}
There is an infinite sequence of abelian groups ${\cal M}_{N}$ and
their subsets ${\cal M}'_{N}$ such that there are no non-trivial
solution of (\ref{sidon}) for elements from these subsets and $n$
grows as $(1+o(1))N^{1/m}$ where $N$ and $n$ are cardinalities of
${\cal M}_{N}$ and ${\cal M}'_{N}$, respectively.
\end{st}

Note that if $N$ is large enough, we may assume $n>
N^{\frac{1}{m}}\cdot \frac{1}{2}$. Consequently, if we take a
specific $n$ large enough, then $N$ can be chosen to be not greater
than $2n ^{m}= O(n^{m})$.

Now notice that non-trivial solutions of (\ref{sidon}) can actually
serve as solutions of many other equations (\ref{gsidon}), see
below.

For example, if in some set ${\cal M}'\subset {\cal M}$ we have
three elements $a,b,c$ forming an arithmetic progression $c+a=2b$,
then this gives rise to a non-tirival solution of (\ref{sidon}): we
set $a_{1}=a,a_{2}=c,a_{3}=b,a_{4}=b$.

Moreover, we have the following obvious
\begin{lem}
Let $k<m$ and let $\alpha_{1},\dots,\alpha_{k}$ be a collection of
integer numbers, $\sum |\alpha_{i}|<m$ and $\sum \alpha_{i}=0$. Then
every solution to

\begin{equation}
\sum_{i=1}^{k} \alpha_{i} b_{i}=0\label{gsidon}
\end{equation}
gives rise to a solution of (\ref{sidon}).
\end{lem}

\begin{proof}
Indeed, collect all positive $\alpha_{i}$ and all negative
$\alpha_{j}$ separately. Let $(b_{1},\dots, b_{k})$ be a solution to
(\ref{gsidon}). For every positive $\alpha_{i}$ we take $\alpha_{i}$
variables from $a_{1},\dots, a_{m}$ to be equal to $b_{i}$, and for
every negative $\alpha_{j}$ we take $-\alpha_{j}$ elements from
$a_{m+1},\dots, a_{2m}$ to be equal to $b_{j}$. It is possible to
choose coordinates of these elements all distinct because $\sum
|\alpha_{i}|<m$. We set the remaining coordinates $a_{k}$ to be $0$.
The claim follows.
\end{proof}

Now, we can modify the set ${\cal M}'$ as follows. Let ${\tilde
{\cal M}'}={\cal M}'+s$, where the addition of $s$ denotes the shift
by a large positive integer number. This number will  be chosen in
such a way that the ratio between the minimal element of ${\cal
M}'+s$ and the maximal element of ${\cal M}'+s$ is strictly greater
than $\frac{m-2}{m}$. We first treat ${\cal M}'$  as a subset of
$\N$. Of course, $s$ grows linearly with respect to $n$. This will
be needed to avoid solutions of equations (\ref{gsidon}) where the
sum of coefficients is non-zero. This leads to an extension ${\tilde
{\cal M}}$ of the group ${\cal M}$ which will be taken to be a
cyclic group $\Z_{2 f(s)}$ where $f(s)$ is larger than the absolute
value of the maximal element of ${\cal M}$ multiplied by $m+1$.

\begin{thm}
Let $\beta_{i}, i=1,\dots, k$ be coefficients such that $\sum
|b_{i}|$ is even and $\sum \beta_{i}\neq 0$. There are no nontrivial
solutions to (\ref{sidon}),(\ref{gsidon}) in ${\tilde {\cal M}}'$;
neither there are any solutions to any of the equations
\begin{equation}
\sum_{i=1}^{k} \beta_{i} c_{i}=0\label{ggsidon}.
\end{equation}
\end{thm}

In other words, having constructed a group and its subset with no
solutions of (\ref{sidon}) and (\ref{gsidon}) with the sum of
coefficients equal to zero, we can easily forbid solutions to all
equations where the sum of coefficients is not equal to zero just by
shifting this set by some function $\mu(m)$ which does not depend on
$n$.

\def\MM{\tilde {\cal M}}
\def\MMM{\tilde {\cal M'}}

Now we are ready to prove the main theorem. First note that any
decomposition of an even $n$ into sum of squares of integer numbers
is a set of numbers which can serve as coefficients of the equations
of of the type (\ref{gsidon}) or (\ref{ggsidon}). Moreover, when
substituting elements from $\MMM$ treated as integer numbers to
(\ref{sidon}),(\ref{gsidon}) or $(\ref{ggsidon})$ we get an integer
number whose absolute value is less than $\lambda(m)\cdot n^{m}$,
where $\lambda(m)$ is some function of $m$ which does not depend on
$n$.

Fix an positive even integer $m$.

Let us take all possible representations of $n$ as the sum of
squares $\sum n_{i}^{2}$ of integer numbers. Such a representation
contains at most $n$ summands, moreover, the sum of these numbers is
even.

Let us choose the set ${\cal M}'$ of cardinality $n$ and the abelian
group ${\cal M}'\supset {\cal M}$ of cardinality $|\MM|= O(n^{m})$
to avoid solutions of (\ref{sidon}),(\ref{gsidon}). By shifting them
by  a large integer number we get the group $\MM$ and the set $\MMM$
in it  avoiding solutions of (\ref{ggsidon}) as well.

Enumerate elements of $\MMM$ by $x_{1},\dots, x_{n}$ and fix the
vector $(x_{1},\dots, x_{n})$ in $\Z^{n}$.

Let us associate with points of $y=\Z^{n}$ integer numbers $\langle
x,y\rangle$. If two points $y,y'$ are at  distance $\sqrt{2m}$ then
$\langle y-y',x\rangle \neq 0$. Indeed, the coordinates of $y-y'$
form a decomposition of $n$ into a sum of squares, and $x_{1},\dots,
x_{n}$ are chosen in such a way that none of the equations of types
(\ref{gsidon}),(\ref{ggsidon}) holds. So, the scalar products are
different.

Besides, $\langle y_{1}',x\rangle$ does not exceed $(max_{ x\in \MM}
|x|)\cdot m$ which grows as $O(n^{m})$.

Thus, taking the residue class of this scalar product modulo
$\lambda(m)\cdot n^{m}+1$, we get a colouring of $\Z^{n}$ with
forbidden distance $\sqrt{2m}$ in the $l_{2}$ norm.

The proof in any other norm $l_{\alpha}$ with the same estimate is
similar.

\section{Lower Estimates for The Chromatic Numbers of Integer Lattices}

We have proved upper polynomial estimates for $\Z^{n}$. Now, we are
going to prove polynomial lower estimates. We shall show that for
many fixed $m$, the exponents $c\cdot n^{m}$ for
$\chi(\Z^{n},\sqrt{2m})$ are optimal.

Let $S$ be a metric space, let $d$ be a critical distance. By a {\em
$(M,D)$-critical configuration} we mean a subset ${\cal M}\subset S$
of cardinality $M$ such that for every subset ${\cal M}'\subset
{\cal M}$ with no two points $a,b\in {\cal M}'$ with critical
distance $l(a,b)=d$, the cardinality $| {\cal M}'|$ is at most $D$.

By the pigeon-hole principle, if there is a critical
$(M,D)$-configuration in $S$ then $\chi(S,d)\ge \chi({\cal M},d)\ge
\frac{M}{D}$.

The lower estimate from $\chi(\Z^{n},2)$ is in fact well-known. We
present it here for consistency.

We shall present a concrete critical configuration. Fix a natural
number $n$, and let $S=\Z^{n}$, ${\cal M}$ be a set of all points
from $\Z^{n}$ having three coordinates equal to $1$ and the others
equal to zero, and let ${\cal M}'$ be a subset of ${\cal M}$ where
no two points are at a distance two. Clearly, $|{\cal
M}|=\left(\begin{array}{c} {n}  \cr {3}\end{array}\right)$. Every
point from ${\cal M}'$ can be considered as a triple of those
coordinates equal to one. Now, the fact that two points $x$ and $y$
from ${\cal M}'$ are at distance not equal to two means that the
corresponding triples are either disjoint or have exactly two common
coordinates. Now, it is easy to see, that the number of such
elements from ${\cal M}'$ can not exceed $n$. So, ${\cal M}$ is an
$(M,D)$-critical configuration, where $M=|{\cal
M}|=\frac{n(n-1)(n-2)}{6}$ and $D=n$.

Thus, the chromatic number for $\Z^{n}$ with critical distance $2$
is greater than or equal to $\frac{(n-1)(n-2)}{6}$.

The methods of finding $(M,D)$-critical configuration are widely
used for establishing lower bounds for the chromatic number of
lattices in arbitrary dimension, the main tool being the well known
Frankl--Wilson theorem \cite{FW} with its further modifications (see
\cite{Rai}).

\begin{thm}[The Frankl--Wilson Theorem]
Let us fix an $n$-element set ${\cal N}=\{1,\dots,n\}$. Let $p$ be a
prime power, and let $a$ be a positive integer number, $a<2p$.
Furthermore, let ${\cal M}$ be a collection of $a$-element subsets
of ${\cal N}$ such that the cardinality of the intersection of any
two of them is not equal to $a-p$. Then $|{\cal
M}|\le\left(\begin{array}{c}{n}\cr {p-1}\end{array}\right)$.
\end{thm}

For modifications of the Frankl-Wilson Theorem, see \cite{Rai3}.

Now, we shall see that for many numbers $2m$ the exponent $m$ in the
upper estimate $n^{m}$ for the chromatic number
$\chi(\Z^{n},\sqrt{2m})$ is optimal for $m$ being a power of a prime
number. Indeed, we consider subsets of ${\cal N}$ as elements of
$\Z^{n}$ with coordinates being equal to $1$ and $0$ ($i$-th
coordinate is equal to $1$ if and only if $i\in {\cal N}$).

The fact given below, written in \cite{Rai}; however, the same
argument was treated as a lower estimate for the chromatic number of
$\R^{n}$, not of $\Z^{n}$.
\begin{thm}
Let $p$ be a power of a prime number. Then
$\chi(\Z^{n},\sqrt{2p})\ge \frac{\left(\begin{array}{c} n \cr
2p-1\end{array}\right)}{\left(\begin{array}{c} n \cr
p-1\end{array}\right)}$.
\end{thm}
\begin{proof}
Indeed, it suffices to take all vectors of length $2p-1$ and forbid
the intersection $p-1$, or, equivalently, forbid the distance
$\sqrt{2p}$. The claim follows.
\end{proof}

Thus, for  $p$ being a power of a prime number, we proved that the
growth of $\chi(\Z^{n},\sqrt{2p})$ is polynomial in $n$ of degree
$p$.


\section{Estimates for rational lattices $\Q^{n}$.}

\begin{thm}
For rational $k$ which can be represented as a sum of two squares of
rational numbers one has
 $\chi(\Q^{2},\sqrt{k})=2$. Otherwise
$\chi(\Q^{2},\sqrt{k})=1$.

Let $m= \frac{p}{q}\sqrt{l}$, where $l$ is an odd number
representable a sum of two squares of integers. Then

$\chi(\Q^{3},m)= 2$;

$\chi(\Q^{4},m)\le 4$.\label{prevth}
\end{thm}

\begin{proof}
The statements about $\Q^{2}$ are obvious. Let us pass to $\Q^{3}$.
Without loss of generality we may assume that $m=\sqrt{l}$.

Let us colour $\Q^{3}$ with two colours, as follows. For a point
$(a,b,c)\in \Q^{3}$, take their minimal common denominator $d$ and
write $a=\frac{a'}{d},b=\frac{b'}{d},c=\frac{c'}{d}$. For colouring,
we take the modulo $2$ residue class of $a'+b'+c'$. Obviously, if
two such points are at distance $\sqrt{l}$, then they have different
colours.

Now, notice the following. If at least one of the numbers $a,b,c\in
\Q$ has even denominator in its reduced fraction, then the sum of
squares of $a,b,c$ can not be an integer. Likewise, if at least one
denominator contains $2^{k}$, then the sum of three squares can not
be a square of a rational number with denominator whose power of $2$
is less than $k$.

Now we use the fact that the points with different exponents of $2$
in denominators ``do not interfere'': the sum of three squares of
integer numbers, at least one of which is odd, can not be even.

Thus, the above colouring can be extended to points of the rational
lattice having coordinates with even denominators. Indeed, we first
shift the initial lattice points with their colours by vectors
$(\frac{1}{2},0,0)$,$(0,\frac{1}{2},0)$,$(0,0,\frac{1}{2})$, then we
shift our colouring by coordinate vectors of length $\frac{1}{4}$,
etc.

To get the estimate for $\Q^{4}$, we shall first colour points with
no coordinate having denominator divisible by four. Every vector $v$
of such sort can be represented as
$(\frac{a}{2s},\frac{b}{2s},\frac{c}{2s},\frac{d}{2s})$, where $s$
is an odd number (possibly, some of $a,b,c,d$ are even). With such a
point we associate the colour $\alpha(v)$ which is equal to residue
class of $a$ modulo $2$.

Let $v_{1}=(\frac{a}{2s},\frac{b}{2s},\frac{c}{2s},\frac{d}{2s})$
and
$v_{2}=(\frac{a'}{2t},\frac{b'}{2t},\frac{c'}{2t},\frac{d'}{2t})$ be
two such vectors; $t,s$ are odd. If $|v_{1}-v_{2}|=l$ then we have
one of two options: either all $a-a'$,$b-b'$,$c-c'$,$d-d'$ are all
odd (in this case $\alpha(v)\neq \alpha(v')$), or all these numbers
are even.

Let us now define $\beta(v)$ as follows. First we define $\beta(v)$
for points from $\Q^{4}$ with all coordinates having odd
denominators: it is just the parity of the sum of numerators. Then
we expand it to points with all denominators of coordinates not
divisible by $4$ by parallel transports by $(\frac{1}{2},0,0,0),
(0,\frac{1}{2},0,0),(0,0,\frac{1}{2},0),(0,0,0,\frac{1}{2})$.

Now, we see that if two vectors $(v_{1},v_{2})$ with denominators of
coordinates not divisible by four are at distance $\sqrt{l}$, then
either $\alpha(v_{1})\neq \alpha(v_{2})$ or $\beta(v_{1})\neq
\beta(v_{2})$. So, we have constructed the four-colouring
$\alpha,\beta$ for all points with denominators of coordinates not
divisible by $4$.

Now, we extend the colouring by shifts by vectors $\frac{1}{2^{l}}$,
where $l\ge 2$. Here we use the fact that the sum of squares of four
integer numbers can not be divisible by $16$ if at least one of them
is odd.

\end{proof}

\begin{thm}
Let $m= \sqrt{2l}\frac{p}{q}$, where $l$ is an odd number such that
$2l$ can be represented as a sum of two integer squares. Then we
have

$\chi(\Q^{4},m)\le 4$, hence, $\chi(\Q^{3},m)\le 4$.
\end{thm}

\begin{proof}
The proof for those points in $\Q^{4}$ for points whose coordinates
have odd denominators, repeats the argument for $\Z^{4}$ from
Theorem \ref{t1k1}: instead of parities of integer numbers, we take
parities of numerators of fractions with odd denominators.

Then this colouring extends to $\Q^{4}$ just by shifting it along
coordinate vectors of lengths $\frac{1}{2^{k}}, k>0$ as in the proof
of Theorem \ref{prevth}.

Here one should take into account that the sum of four squares of
integer numbers can not be divisible by $8$ if at least one of these
numbers is odd.
\end{proof}

\section{Colourings of Some Finite Graphs}

Let us consider the fields $\Z_{3}$ and $\Z_{5}$; we shall construct
graphs $\Z_{3}^{n}$ and $\Z_{5}^{m}$, where for the (pseudo)metric
we take the $l_{2}$-metric taken modulo $3$ (resp., modulo $5$).
\begin{thm}
$\chi(\Z_{3}^{n},1)\le c (\sqrt[3]{9})^{n}.$\label{3grid}
\end{thm}

\begin{proof}
The proof is by induction on the dimension $n$. It suffices for us
to prove that for $n=2+3k$ we have $\chi(\Z_{3}^{n})\le 3^{2k+1}$
for positive integers $k$.

For $\Z_{3}^{2}$, let us use three colours to colour $9$ points: we
just take the colour to be the modulo three residue class of the sum
of coordinates.

Now, assume we have a proper colouring of $\Z_{3}^{2+3k}$; let us
colour $\Z_{3}^{5+3k}$ as follows. We colour the first $2+3k$
coordinates by using $3^{2k+1}$ colours, and take $9$ colours for
$\Z_{3}^{3}$.  The colour for $\Z_{3}^{5+3k}$ will consist of two
components, the one for the first $2+3k$ coordinates, and the one
for the last three coordinates. The last component will have $9$
colours, namely, for $(a,b,c)\in \Z_{3}^{3}$ we take the colour to
be $(b-a,c-a)\in \Z_{3}\oplus \Z_{3}$.  If for two points $(a,b,c)$
and $(a',b',c')$ from $\Z_{3}^{3}$ we have $b-a\equiv b'-a' \mbox{
mod } 3$ and $c-a\equiv c'-a'\mbox{ mod } 3$, then these points
either coincide if $a=a'$, or these points $(a,b,c)$ and
$(a',b',c')$ are at a distance three if $a\neq a'$. In any case,
this colouring forbids distances congruent to $1$ and $2$ modulo
$3$.

We claim that for such a colouring of $\Z_{3}^{5+3k}$ no two points
at distance congruent to $1$ modulo $3$ have the same colour.
Indeed, for two points $x,y\in \Z_{3}^{5+3k}$ at distance congruent
to $1$ modulo $3$, either the distance between the projections to
the first $2+3k$ coordinates is congruent to $1$ modulo $3$, or the
distance between the projections to the last $3$ coordinates is not
congruent to $0$ modulo $3$. In the first case, the colours of these
points have different first component; in the second case, the
colours have different second component.

This completes the induction step.
\end{proof}

\begin{thm}
$\chi(\Z_{5}^{n},1)\le c' (\sqrt{5})^{n}.$\label{5grid}
\end{thm}

\begin{proof}
We proceed in a way similar to the above. We establish the induction
base by colouring  $\Z_{5}$ with five different colours; then we
colour $\Z_{5}^{2}$ in five colours so that two points share one
colour whenever they are at distance congruent to $0$ modulo five.
Namely, for $(a,b)\in \Z_{5}^{2}$ we take the colour to be $a-2b
\mbox{ mod 5}$. Then we proceed by induction: for every two new
coordinates we need to multiply the number of colours by $5$, and
the result follows.
\end{proof}

\begin{re}
The above estimates remain true if instead of $1$ we take any
forbidden distance congruent to $1$ modulo $3$ (congruent to $1$
modulo $5$, respectively).
\end{re}

\section{Estimates for Lattices over Some Algebraic Extensions of $\Z$}

Now, let $p_{1},\dots, p_{k}$ be a set of integers. By
$\Q_{p_{1},p_{2},\dots, p_{k}}$ we shall denote the ring of rational
numbers with denominators coprime with $p_{1} \dots
 p_{k}$. By $\Q_{odd}$ we mean the set of rational numbers with odd
denominators only.

\begin{thm}
$\chi(\Q_{odd}^{n},1)=2$. Moreover, for every extension ${\cal K}$
of the ring of integers for which there exists a homomorphism ${\cal
K}\to \Z_{2}$ we have $\chi({\cal K}^{n},1)=2$.
\end{thm}

\begin{proof}
Indeed, assume all denominators of coordinates are odd. Then for the
colouring we take the set of modulo 2 residue classes of numerators.
\end{proof}

\begin{thm}
$\chi(\Q_{3}^{n},1)\le c (\sqrt[3]{9})^{n},$ where $c$ is some
universal constant. The same remains true if one replaces $\Q_{3}$
with any subring of $\R$ admitting a homomorphism to $\Z_{3}$.
\end{thm}

\begin{thm}
 $\chi(\Q_{5}^{n},1)\le c' (\sqrt{5})^{n}$, where $c'$ is some universal constant. Moreover,
 the same is true if one replaces
$\Q_{5}$ with a subring of the ring of integers admitting a
homomorphism to $\Z_{5}$.
\end{thm}

The two last theorems easily follow from Theorems \ref{3grid} and
\ref{5grid}. The idea is to take the coordinates
 $(x_{1},\dots, x_{n})$ modulo $3$ (repectively, modulo $5$) and to use the estimate for
$\chi(\Z_{3}^{n})$ (respectively, for, $\chi(\Z_{5}^{n})$). Here
``taking the residue class'' means considering the corresponding
ring homomorphism.

\section{Some Open Problems}

We conjecture that all possible critical distances where the
chromatic numbers for $\Z^{3}$ is equal to $3$, are those of form
$2^{k}\sqrt{12l+10}$; besides, we conjecture that the chromatic
number $3$ never occurs for integer lattices in higher dimensions.


The best known upper asymptotic estimate for the chromatic number of
rational spaces still remains the same as for Euclidean spaces of
the same dimension: it is $(3+o(1))^{n}$, see \cite{LR}; the methods
of obtaining these estimates are based on some Vorono\"i tilings of
Euclidean spaces; in other words, these known upper estimates come
from tilings of the Euclidean spaces into smaller parts.

Moreover, every lower bound for $\chi(\Q^{n},\sqrt{d})$ for some
concrete $n,\sqrt{d}$ comes from a concrete finite graph $\Gamma$ in
$\Q^{n}$ with critical distance $\sqrt{d}$.

The possibility to get an exact estimate from a {\em finite} graph
is exactly the  de Bruijn--Erd\H{o}s theorem \cite{BE}. If we
consider such a graph for $\Q^{n}$ and take the common denominator
$D$ of all coordinates of all points of this graph, we get a
homothetic graph $D\Gamma$ in $\Z^{n}$ with critical distance
$D\alpha$. So, all lower estimates for rational lattices actually
come from integer lattices.

It would be interesting to apply the argument of the present paper
to obtain sharper estimates for $\Q^{n}$. The direct approach fails
because when taking some concrete forbidden distance, one will have
to take the maximum over all estimates for $D\alpha$ which tends to
infinity as $n$ tends to infinity.

Our estimates for lattices with rational coordinates with some
restrictions on denominators are somewhat better but use principally
different ideas: some number theoretic properties of $\mbox{ modulo
} p$ reductions. It would be very interesting to get other estimates
for $\chi(\Q^{n})$ by combining the two approaches: the one from the
present paper and the one using Vorono\"i tilings.

 We have found upper estimates for $\chi(\Z^{n},\sqrt{d})$ for
every fixed $d$ as $n$ tends to infinity. If we fix a concrete $n$,
then we have estimates for odd d: $2$ colours, for $d$ not divisible
by $3$: $c_{1}\cdot (\sqrt[3]{9})^{n}$ colours, and for $d$ not
divisible by $5$ we get $c_{2}\cdot (\sqrt{5})^{n}$ colours. All
these upper bounds are better than the best known estimates for
rational lattices, which is $(3+o(1))^{n}$. So, it would be
interesting to find an upper bound for
$\mbox{max}_{30|d}(\chi(\Z^{n},\sqrt{d}))$, where the maximum is
taken over all $d$ divisible by $30$. Possibly, there is a way to
elaborate similar methods for other prime numbers; however, an
argument for numbers whose denominators are not divisible by $7$
similar to those given for numbers whose denominators are coprime
with $2,3,5$ will give an estimate which is worse than the
well-known one for rational lattices.


\begin{thebibliography}{100}


\bibitem{BE} de Bruijn, N. G.; Erd\H{o}s, P. (1951), A colour problem for
infinite graphs and a problem in the theory of relations, Nederl.
Akad. Wetensch. Proc. Ser. A 54: pp. 371–-373.

\bibitem{BMP} Brass, P., Moser, L., Pach, J., (2005), {\em Research
Problems in Discrete Geometry}, Springer.

\bibitem{BP} Benda, M. Perles,M. (2000), Introduction to Colorings of Metric
Spaces, {\em Geombinatorics}, 9, pp. 111-126.




\bibitem{Cibulka} Cibulka, J. (2008), On the Chromatic Numbers of Real
and Rational Spaces,  {\em Geombinatorics}, 18, pp. 53-–65.

\bibitem{FK} F\H{u}redi, Z., Kang, J.-H. (2004), Distance Graphs on $\Z^{n}$
with $l_{1}$-norm,{\em Theoretical Computer Science} 319, pp.
357-366.


\bibitem{FW} Frankl,P., Wilson,R.M. (1981), Intersection Theorems with Geometric
Consequences, {\em Combinatorica},{\bf 1}, pp. 357-368.



\bibitem{Kup} Kupavskii, A.B. (2011), On the colouring of spheres embedded in ${\mathbb
R}^n$, {\em Sbornik: Mathematics}, 202(6): pp. 859-886.

\bibitem{LR} Larman D.G., Rogers,C.A. (1972), The Realization of Distances within
Sets in Euclidean Spaces, {\em Mathematica}, London. (19), pp. 1-24.




\bibitem{MM} Moser,L. and Moser, W. (1961), Solution to Problem 10, {\em
Canad.Math.Bull}, {\bf 4}, pp. 187-189.

\bibitem{OB} O'Bryant, K. (2011), A Complete Annotated Bibliography of Work Related to Sidon
Sequences, arXiv:math.NT$\slash$0407.117


\bibitem{Rai} Raigorodskii A.M. (2001), Borsuk's Problem and the Chromatic
Numbers of Some Metric Spaces, {\em Russ. Math. Surv.}, {\bf 56}
(1), pp. 103-139.

\bibitem{Rai2} Raigorodskii A.M. (2004), The chromatic number of a space with the metric
$l_q$, {\em Russian Mathematical Surveys}, {\bf 59} (5), p. 973


\bibitem{Rai3} Raigorodskii A.M., (2007), Lineino algebraicheskie metody
v kombinatorike (Linear algebraic methods in combinatorics, in
Russian), MCCME.

\bibitem{Raiskii} Raiskii, D.E.(1970), Realization of all distances in a decomposition
of $\R^{n}$ into $n+1$ parts, {\em Mathematical Notes}, {\bf 7}, pp.
194-196.

\bibitem{Ruzsa1} Ruzsa, I. (1993), Solving a linear equation in a set
of integers II, {\em Acta Arithmetica}, LXV.3, pp. 259-282.

\bibitem{Ruzsa2} Ruzsa, I. (1995), Solving a linear equation in a set
of integers II, {\em Acta Arithmetica}, LXXII.4, pp. 385-397.

\end{thebibliography}
\end{document}